\documentclass[12pt,reqno]{amsart}
\usepackage[bookmarksnumbered,colorlinks, plainpages]{hyperref}
\hypersetup{colorlinks=true,linkcolor=red, anchorcolor=green, citecolor=cyan, urlcolor=red, filecolor=magenta, pdftoolbar=true}

\usepackage{subfig}
\usepackage{tikz}
\usetikzlibrary{arrows}

\usepackage{array}

\usepackage{amsmath, amsthm, amscd, amsfonts, amssymb, graphicx, color, enumerate, xcolor, mathrsfs, latexsym}

\newcommand{\A}{\mathcal{A}}

\newcommand{\I}{\mathcal{I}}
\renewcommand{\H}{\mathcal{H}}
\newcommand{\W}{\mathcal{W}}

\newcommand{\Z}{\mathbb{Z}}

\renewcommand{\d}{\mathtt{d}}
\newcommand{\h}{\mathtt{h}}
\renewcommand{\r}{\mathtt{r}}
\renewcommand{\u}{\mathtt{u}}

\renewcommand{\S}{\mathbf{S}}

\renewcommand{\leq}{\leqslant}
\renewcommand{\geq}{\geqslant}

\newcommand{\M}{\mathcal{M}}

\renewcommand{\leq}{\leqslant}
\renewcommand{\geq}{\geqslant}

\numberwithin{equation}{section}

\newtheorem{theorem}{Theorem}[section]
\newtheorem{lemma}[theorem]{Lemma}

\newtheorem{corollary}[theorem]{Corollary}

\theoremstyle{definition}

\newtheorem{example}[theorem]{Example}

\begin{document}
\title{Lattice paths inside a table II}

\author[D. Yaqubi]{Daniel Yaqubi}
\address{Department of Pure Mathematics, Ferdowsi University of Mashhad, P. O. Box 1159, Mashhad 91775, Iran}
\email{daniel\_yaqubi@yahoo.es}
\author[M. Farrokhi D. G.]{Mohammad Farrokhi D. G.}
\address{Institute for Advanced Studies in Basic Sciences (IASBS), and the Center for Research in Basic Sciences and Contemporary Technologies, IASBS, P.O.Box 45195-1159, Zanjan 66731-45137, Iran}
\email{m.farrokhi.d.g@gmail.com\\farrokhi@iasbs.ac.ir}

\author[M. Zamani K.]{Mohamad Zamani K.}
\address{Amirkabir University of Technology (Tehran Polytechnique), Tehran, Iran}
\email{m.zamani28@aut.ac.ir}


\maketitle
\begin{abstract}
Consider an $m\times n$ table $T$ and latices paths $\nu_1,\ldots,\nu_k$ in $T$ such that each step $\nu_{i+1}-\nu_i=(1,1)$, $(1,0)$ or $(1,-1)$. The number of paths from the $(1,i)$-cell (resp. first column) to the $(s,t)$-cell is denoted by $\mathcal{C}^i(s,t)$ (resp. $\mathcal{C}(s,t)$). Also, the number of all paths form the first column to the last column is denoted by $\I_m(n)$. We give explicit formulas for the numbers $\mathcal{C}^1(s,t)$ and $\mathcal{C}(s,t)$.
\end{abstract}
\section{Introduction}
A lattice path in $\Z^2$ is the drawing in $\Z^2$ of a sum of vectors from a fixed finite subset $S$ of $\Z^2$, starting from a given point, say $(0,0)$ of $\Z^2$. A typical problem in lattice paths is the enumeration of all $\S$-lattice paths (lattice paths with respect to the set $\S$) with a given initial and terminal point satisfying possibly some further constraints. A nontrivial simple case is the problem of finding the number of lattice paths starting from the origin $(0,0)$ and ending at a point $(m,n)$ using only right step $(1,0)$ and up step $(0,1)$ (i.e., $\S=\{(1,0),(0,1)\}$). The number of such paths are known to be the the binomial coefficient $\binom{m+n}{n}$. Yet another example, known as the ballot problem, is to find the number of lattice paths from $(1,0)$ to $(m,n)$ with $m>n$, using the same steps as above, that never touch the line $y=x$. The number of such paths, known as ballot number, equals $\frac{m-n}{m+n}\binom{m+n}{n}$. In the special case where $m=n+1$, the ballot number 	is indeed the Catalan number $C_n$.

Let $T_{m,n}$ denote the $m\times n$ table in the plane and $(x,y)$ be the cell in the columns $x$ and row $y$ (and refer to it as the $(x,y)$-cell). The set of lattice paths from the $(i,j)$-cell to the $(s,t)$-cell, with steps belonging to a finite set $\S$, is denoted by $L((i,j)\to (s,t);\S)$, and the number of those paths is denoted by  $L((i,j)\to (s,t);\S)$, where $1\leq i,s\leq m$ and $1\leq j,t\leq n$. We put $|L((i,j)\to (s,t);\S)|= l((i,j)\to (s,t);\S)$ which means the number of all lattice paths from the $(i,j)$-cell to the $(s,t)$-cell.

 Throughout this paper, for the table $T_{m,n}$, we set $\S=\{(1,1),(1,0),(1,-1)\}$, and the corresponding lattice paths starting from the first column and ending at the last column are called \textit{perfect lattice paths}. The number of all perfect lattice paths is denoted by $\I_m(n)$, that is,
\[\I_m(n)=\sum_{i,j=1}^m l((i,j)\to (s,t);\S).\]
The values of $\I_m(n)$ is OEIS sequence \href{https://oeis.org/A081113 }{A081113 } and \href{https://oeis.org/A296449}{A296449}. 

Sometimes it is more convenient to name each step of lattice paths by a letter, and hence every lattice path will be encoded as a \textit{lattice word}. We label the steps of the set $\S=\{(1,1),(1,0),(1,-1)\}$ by letters $\u=(1,1)$, $\r=(1,0)$, and $\d=(1,-1)$; also if $\h$ is a letter of the word $\W$, order or size of $\h$ in $\W$ is the number of times the letter $\h$ appears in the word $\W$ and it is denoted by $|\h|=|\h|_\W$.

In this paper, by using ballot numbers, we give explicit formulas for the numbers  $\mathcal{C}^1(s,t)$ and $\mathcal{C}(s,t)$ where are defined in the section $2$. We closed this paper by calculating the number of perfect lattice paths without restrictions in the table $T$.  
\section{Computing $\I_m(n)$ in special cases}
In this section, we give formulas for the number $\I_m(n)$ in the cases where $n+1\leq m\leq 2n$ and $2n\leq m$. To achieve this goal, we must to recall some further notations from \cite{dy-mfdg-hgz}. Let $T$ be the $m\times n$ table and $\S=\{(1,0),(1,1),(1,-1)\}$. The number of lattice paths from the $(1,i)$-cell to the $(s,t)$-cell is denoted by $\mathcal{C}^i(s,t)$. Indeed, $\mathcal{C}^i(s,t)=l((1,i)\to (s,t);\S)$. Table \ref{Table:1} illustrates the values of $\mathcal{C}^8(s,t)$ for $1\leq s,t\leq8$. 
Also, the number of lattice paths from the first column to the $(s,t)$-cell is denoted by $\mathcal{C}_{m,n}(s,t)$, that is,
\[\mathcal{C}_{m,n}(s,t)=\sum_{i=1}^m\mathcal{C}^i(s,t).\]
To avoid confusion we may use simply notation $\mathcal{C}(s,t)$ for $\mathcal{C}_{m,n}(s,t)$. For the square table $T_{n,n}$, some values of $\mathcal{C}(n,n)$ is $1, 1, 2, 5, 13, 35, 96, 267, 750,\ldots$ are given in the OEIS sequence \href{https://oeis.org/A005773}{A005773}. D. Gouyou-Beauchamps, G. Viennot show that $\mathcal{C}(n,n)$ enumerate the number of directed animals of size n (or directed n-ominoes in standard position) \cite{Vie}.Clearly, $\I_m(n)$ is the number of words $a_1a_2\ldots a_{n-1}a_n$ ($a_i\in\{1,\ldots,m\}$) such that $|a_{i+1}-a_i|\leq1$ for all $i=1,\ldots,n-1$. Figure \ref{figure:1} shows perfect lattice paths in $T_{2,3}$ and the corresponding words, where the $i^{th}$ letter indicates the rows whose $i^{th}$ point of the paths belongs to.

\begin{figure}[h!]
\begin{tikzpicture}[scale=1]
\draw[step=1cm] (0,0) grid (3,2);
\draw[o-stealth] (0.5, 0.5) -- (1.5, 0.5);
\draw[o-stealth] (1.5, 0.5) -- (2.5, 0.5);
\end{tikzpicture}
\begin{tikzpicture}[scale=1]
\draw[step=1cm] (0,0) grid (3,2);
\draw[o-stealth] (0.5, 0.5) -- (1.5, 0.5);
\draw[o-stealth] (1.5, 0.5) -- (2.5, 1.5);
\end{tikzpicture}
\begin{tikzpicture}[scale=1]
\draw[step=1cm] (0,0) grid (3,2);
\draw[o-stealth] (0.5, 0.5) -- (1.5, 1.5);
\draw[o-stealth] (1.5, 1.5) -- (2.5, 0.5);
\end{tikzpicture}
\begin{tikzpicture}[scale=1]
\draw[step=1cm] (0,0) grid (3,2);
\draw[o-stealth] (0.5, 0.5) -- (1.5, 1.5);
\draw[o-stealth] (1.5, 1.5) -- (2.5, 1.5);
\end{tikzpicture}
\\
$111$\hspace{2.55cm}$112$\hspace{2.55cm}$121$\hspace{2.55cm}$122$
\\\ \\
\begin{tikzpicture}[scale=1]
\draw[step=1cm] (0,0) grid (3,2);
\draw[o-stealth] (0.5, 1.5) -- (1.5, 1.5);
\draw[o-stealth] (1.5, 1.5) -- (2.5, 1.5);
\end{tikzpicture}
\begin{tikzpicture}[scale=1]
\draw[step=1cm] (0,0) grid (3,2);
\draw[o-stealth] (0.5, 1.5) -- (1.5, 1.5);
\draw[o-stealth] (1.5, 1.5) -- (2.5, 0.5);
\end{tikzpicture}
\begin{tikzpicture}[scale=1]
\draw[step=1cm] (0,0) grid (3,2);
\draw[o-stealth] (0.5, 1.5) -- (1.5, 0.5);
\draw[o-stealth] (1.5, 0.5) -- (2.5, 1.5);
\end{tikzpicture}
\begin{tikzpicture}[scale=1]
\draw[step=1cm] (0,0) grid (3,2);
\draw[o-stealth] (0.5, 1.5) -- (1.5, 0.5);
\draw[o-stealth] (1.5, 0.5) -- (2.5, 0.5);
\end{tikzpicture}
$222$\hspace{2.55cm}$221$\hspace{2.55cm}$212$\hspace{2.55cm}$211$
\caption{Associated words for lattice paths in $T_{2,3}$}
\label{figure:1}
\end{figure}
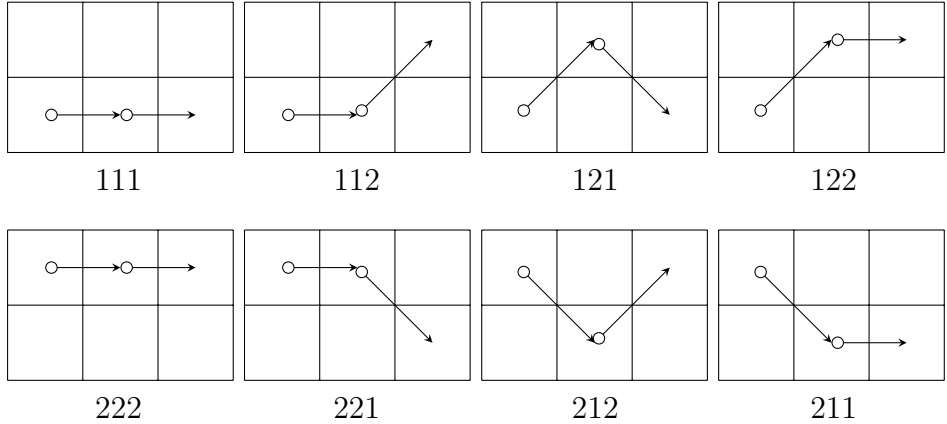 

In what follows, the number of lattice paths from $(1,1)$-cell to $(s,t)$-cell ($1\leq s\leq n$ and $1\leq t\leq m$), using just the two steps $(1,1)$ and $(1,-1)$, is denoted by $\A(s,t)$. In other words, $\A(s,t)=l(1,1;s,t:\S')$, where $\S'=\{(1,1),(1,-1)\}$. Table \ref{Table:1} illustrates the values of $\A(s,t)$ for $1\leq s,t\leq8$. Clearly $\A(s,t)=0$ for $s<t$, and that $\A(s,t)$ is the number lattice paths from the $(1,1)$-cell to $(s,t)$-cell not sliding above the line $y=x$. One observe that $\A(s,t)=0$ if $s,t$ have distinct parities as the paths counted by $\A(s,t)$ begins from $(1,1)$ and every step in $\S'$ keeps the parities of entries so that such paths never meet $(s,t)$-cells with $(s,t)$ having distinct parities. Using the symbols $\u$ and $\d$, the number $\A(s,t)$ counts the words of length $s-1$ on $\{\u,\d\}$ whose all initial subwords have more or equal $\u$ than $\mathcal{C}$. For example, Table \ref{Table:1} tells us $\A(6,1)=5$ and the corresponding five words are
\[\u\u\u\d\d,\ \u\u\d\u\d,\ \u\d\u\d\u,\ \u\u\d\d\u,\ \u\d\u\u\d.\]
Analogous to $\A(s,t)$, the number $\mathcal{C}^1(s,t)$ counts the words $a_1a_2\ldots a_i$ with $1\leq a_i\leq t$ such that $|a_{i+1}-a_i|\leq 1$ for all $1\leq i\leq s-1$. In other words, $\mathcal{C}^1(s,t)$ counts the number of words of length $s-1$ on $\{\u,\r,\d\}$ whose all initial subwords have more or equal $\u$ than $\d$. For example, Table \ref{Table:1} shows that $\mathcal{C}^1(4,1)=5$, and the corresponding five words are
\[\u\u\d,\ \u\r\r,\ \r\r\u,\ \u\d\u,\ \r\u\r.\]
\begin{table}[h!]
\footnotesize{
\begin{center}
\begin{tabular}{|c|c|c|c|c|c|c|c|}
\hline
&&&&&&&$1$\\
\hline 
&&&&&&$1$&$7$\\
\hline
&&&&&$1$&$6$&$27$\\
\hline 
&&&&$1$&$5$&$20$&$70$\\
\hline
&&&$1$&$4$&$14$&$44$&$133$\\
\hline
&&$1$&$3$&$9$&$25$&$69$&$189$\\
\hline
&$1$&$2$&$5$&$12$&$30$&$76$&$196$\\
\hline
$1$&$1$&$2$&$4$&$9$&$21$&$51$&$127$\\
\hline
\end{tabular}
\begin{tabular}{|c|c|c|c|c|c|c|c|}
\hline
&&&&&&&$1$\\
\hline 
&&&&&&$1$&$0$\\
\hline
&&&&&$1$&$0$&$6$\\
\hline 
&&&&$1$&$0$&$5$&$0$\\
\hline
&&&$1$&$0$&$4$&$0$&$14$\\
\hline
&&$1$&$0$&$3$&$0$&$9$&$0$\\
\hline
&$1$&$0$&$2$&$0$&$5$&$0$&$14$\\
\hline
$1$&$0$&$1$&$0$&$2$&$0$&$5$&$0$\\
\hline
\end{tabular}
\end{center}}
\caption{Values of $\mathcal{C}^1(s,t)$ (left), and values of $\A(s,t)$ (right)}
\label{Table:1}
\end{table}
\begin{theorem}\label{D^1(s,t) by A(x,y)}
For all $1\leq s,t\leq m$, we have
\[\mathcal{C}^1(s,t)=\sum_{i=0}^{\lfloor\frac{s-t}{2}\rfloor}\binom{s-1}{s-t-2i}\A(t+2i,t).\]
\end{theorem}
\begin{proof}
Let $P$ be a lattice path starting from the $(1,1)$-cell and ending at the $(s,t)$ with three steps $\u,\r,\d$ that never slides above the line $y=x$. Clearly, $|\u|_P-|\d|_P=t-1$. Since the number of steps to reach the $(s,t)$-cell is $s-1$ steps, we must have 
\[|\r|_P=s-1-|\u|_P-|\d|_P=s-t-2|\d|_P.\]
Omitting all the $\r$ steps from $P$ yields a path $P'$ from $(1,1)$-cell to the $(t+2|\d|_P,t)$-cell using only $\u$ and $\d$ steps. For any such a path $P'$ one can recover 
\[\binom{|\u|_{P'}+|\d|_{P'}+s-t-2|\d|_{P'}}{\underbrace{s-t-2|\d|_{P'}}_{|\r|_P}}=\binom{s-1}{s-t-2|\d|_{P'}}\]
paths $P$ by inserting right steps $\r$ among $\u$ and $\d$ steps of $P'$, from which the result follows.
\end{proof}
\begin{example}
From Table \ref{Table:1} we read $\mathcal{C}^1(8,4)=133$. Using theorem \ref{D^1(s,t) by A(x,y)}, we can compute $\mathcal{C}^1(8,4)$ alternately as
\begin{eqnarray*}
\mathcal{C}^1(8,4)&=&\sum_{i=0}^{\lfloor\frac{8-4}{2}\rfloor}\binom{8-1}{8-4-2i}\A(2i+4,4)\\
&=&{7\mathcal{C}hoose 4}\A(4,4)+{7\mathcal{C}hoose 2}\A(6,4)+{7\mathcal{C}hoose 0}\A(8,4)\\
&=&35\times 1+21\times 4+1\times 14=133.
\end{eqnarray*}
\end{example}

The numbers $\A(s,t)$ are indeed computed as in the ballot problem were the paths can touch the $y=x$ line but never go above it. The number of such ballot paths from $(1,0)$ to $(m,n)$ is $\frac{m-n+1}{m+1}\binom{m+n}{m}$. Recall that $\A(s,t)$ is the number of words $\W$ of length $s-1$ on $\{\u,\d\}$ with more or equal $u$ than $d$ in any initial subword, hence $\A(s,t)$ is equal to the above number with $m:=|\u|_\W$ and $n:=|\d|_\W$. Now since $|\u|_\W+|\d|_W=s-1$ and $|\u|_\W-|\d|_\W=t-1$, it follows that $m=(s+t)/2-1$ and $n=(s-t)/2$. Hence we obtain the following
\begin{lemma}\label{A}
Inside the $n\times n$ table, we have 
\[\A(s,t)=\frac{2t}{s+t}\binom{s-1}{\frac{s-t}{2}}.\]
for all $1\leq s,t\leq n$.
\end{lemma}
\begin{corollary}
Inside the $n\times n$ table, we have 
\[\mathcal{C}^1(s,t)=\sum_{i=0}^{\lfloor\frac{s-t}{2}\rfloor}\frac{t}{t+i}\binom{s-1}{s-t-2i}\binom{t+2i-1}{i}.\]
for all $1\leq s,t\leq n$.
\end{corollary}

In \cite{dy-mfdg-hgz}, we have computed the number $\I_n(n)$ for all $n\geq1$. In what follows, we shall give formulas for $\I_m(n)$, where $n+1\leq m\leq 2n$.
To achieve this goal, we use the numbers $\H(s,t)$ inside the $m\times n$ table defined as
\[\H(s,t)=\sum_{i=1}^t\mathcal{C}^1(s,i),\]
where $1\leq s\leq n$ and $1\leq t\leq m$. Table \ref{Table:2} illustrates some values of $\H(s,s)$. One observe that $\H(s,s)=\mathcal{C}(s,s)$ for all $s\leq m=5$.

\begin{table}[h!]
\centering
\begin{tabular}{|c|c|c|c|c|c|c|c|c|c|c|}
\hline
&&&&&$1$&$5$&$19$&$63$&$195$&$579$\\
\hline
&&&&$1$&$4$&$14$&$44$&$133$&$384$&$1096$\\
\hline
&&&$1$&$3$&$9$&$25$&$69$&$189$&$517$&$1413$\\
\hline
&&$1$&$2$&$5$&$12$&$30$&$76$&$196$&$512$&$1352$\\
\hline
&$1$&$1$&$2$&$4$&$9$&$21$&$51$&$127$&$323$&$835$\\
\hline\hline
$\H(s,s)$&$1$&$2$&$5$&$13$&$36$&$95$&$259$&$708$&$1931$&$5275$\\
\hline
\end{tabular}
\caption{Some values of $\H(s,s)$ for $T_{5,10}$}
\label{Table:2}
\end{table} 
\begin{lemma}\label{H(n,m) by D^1(i,m)}
Inside the $m\times n$ table with $m\leq n\leq 2m$, we have
\[\H(n,m)=\mathcal{C}(n,n)-\sum_{i=m}^{n-1}3^{n-i-1}\mathcal{C}^1(i,m).\]
\end{lemma}
\begin{proof}
Consider the $m\times n$ table $T$ as the subtable of the $n\times n$ table $T'$ with $T$ in the bottom. We know that $\mathcal{C}(n,n)=\H(n,n)$ is the number of all perfect lattice paths from the $(1,1)$-cell to the last columns. However, some lattice paths leave $T$ in rows $m+1,m+2,\ldots n$ of $T'$. We shall count the number of such paths. Consider the column $i$ a path starting from $(1,1)$ left the table $T$ for the first times. Clearly, $m+1\leq i\leq n$. The number of such paths is $\mathcal{C}^1(i-1,m)$ and the number of paths starting from $(i,m+1)$-cell and ending at the last column in simply $3^{n-i}$. Thus the number of such paths leaving $T$ is equal to
\[\sum_{i=m+1}^n3^{n-i}\mathcal{C}^1(i-1,m)=\sum_{i=m}^{n-1}3^{n-i-1}\mathcal{C}^1(i,m),\]
from which the result follows.
\end{proof}
\begin{example}
Using Lemma \ref{H(n,m) by D^1(i,m)}, we can calculate $\H(9,5)$ as
\begin{align*}
\H(9,5)=&\mathcal{C}(9,9)-\sum_{i=5}^{8}3^{8-i}\mathcal{C}^1(i,5)\\
=&2123-\big(3^{8-5}\mathcal{C}^1(5,5)+3^{8-6}\mathcal{C}^1(6,5)\\
&\quad\quad\quad\quad\quad\quad+3^{8-7}\mathcal{C}^1(7,5)+3^{8-8}\mathcal{C}^1(8,5)\big)\\
=&2123-\big(27\times 1+9\times 5+3\times 19+1\times 63)=1931.
\end{align*}
\end{example}
\begin{lemma}\label{D^1(s,m) by D^1(m,i)D^1(s-m+1,m-i+1)}
Inside the $m\times n$ table, we have
\[\mathcal{C}^1(n,m)=\sum_{i=1}^m\mathcal{C}^1(s,i)\times\mathcal{C}^1(n-s+1,m-i+1).\]
for all $1\leq s\leq n$.
\end{lemma}
\begin{proof}
Every path from the $(1,1)$-cell to $(n,m)$-cell crosses the $s^{th}$ column at some row, say $i$. The number of such paths equals the number $\mathcal{C}^1(s,i)$ of paths from the $(1,1)$-cell to $(s,i)$-cell multiplied by the number $\mathcal{C}^1(n-s+1,m-i+1)$ of paths from the $(n,m)$-cell to $(s,i)$-cell (in reversed direction). The result follows.
\end{proof}

\begin{example}
Table \ref{Table:1} shows that $\mathcal{C}^1(9,5)=195$. Lemma \ref{D^1(s,m) by D^1(m,i)D^1(s-m+1,m-i+1)} gives an alternate way to compute $\mathcal{C}^1(9,5)$ as in the following:
\begin{align*}
\mathcal{C}^1(9,5)=&\sum_{i=1}^5\mathcal{C}^1(5,i)\mathcal{C}^1(9-5+1,5-i+1)\\
=&\mathcal{C}^1(5,1)\mathcal{C}^1(5,5)+\mathcal{C}^1(5,2)\mathcal{C}^1(5,4)+\mathcal{C}^1(5,3)\mathcal{C}^1(5,3)\\
&+\mathcal{C}^1(5,4)\mathcal{C}^1(5,2)+\mathcal{C}^1(5,5)\mathcal{C}^1(5,1)\\
=&9\times 1+12\times 4+9\times 9+4\times 12+1\times 9=195.
\end{align*}
\end{example}

\begin{lemma}\label{D(s,t) by D^1(i,t) and D^1(i,m+1-t)}
Inside the $m\times n$ table, we have
\[\mathcal{C}(s,t)=3^{s-1}-\sum_{i=t+1}^{s-1}3^{s-i-1}\mathcal{C}^1(i,t)-\sum_{i=m+2-t}^{s-1}3^{s-i-1}\mathcal{C}^1(i,m+1-t)\]
for all $s\leq n+2$.
\end{lemma}
\begin{proof}
Allowing the paths leak out of the $m\times n$ table $T$, the number of all perfect lattice paths from the first column to the $(s,t)$-cell is $3^{s-1}$ minus those paths leaving $T$ at some step. Suppose a path leaves $T$ for the last times at $(0,i)$-cell. Then $i=1,\ldots,s-t-1$ and the number of such paths equals $3^{i-1}\mathcal{C}^1(s-i,t)$. Analogously, the number of paths leaving $T$ for the last time at $(i,m+1)$-cell is $3^{i-1}\mathcal{C}^1(s-i,m+1-t)$, and that $0\leq i\leq s-(m+1-t)$. Hence the result follows by changing $i$ to $s-i$.
\end{proof}
In the \cite{dy-mfdg-hgz} the authors obtained the following relation for $\mathcal{C}_n$ in terms of Motzkin numbers $M_n$.
\begin{lemma}\label{D(s,1)}
Inside the square $n\times n$ table we have
\[\mathcal{C}_n=3\mathcal{C}_{n-1}-\M_{n-2}.\]
\end{lemma}
Utilizing the above recurrence relation, we prove the following theorem.
\begin{theorem}
Let $T$ be a $n\times n$ table. Then
\[\mathcal{C}_{n-1}=\sum_{i=3}^{n+}M_{i-3}\mathcal{C}_{n-i+1},\]
where $M_i$ is the $i^{th}$ -Motzkin number.
\end{theorem}\label{C(n,n)}
\begin{proof}
$\mathcal{C}_n$ is the number of perfect lattice paths from the cell $(1,1)$ to $(n,n)$. Consider $P_n\in \mathcal{C}_n$ is the paths of length $n$ in the following table. If the first step of $P_n$ is $(1,0)$ , the number of perfect lattice paths from this cell to the cell $(n-1,n-1)$ is $\mathcal{C}_{n-1}$. Now, let the first step of $P_n$ is $(1,1)$. There are two cases for the number of lattice paths from the cell $(1,1)$ to the cell $(n,n)$.  First, for the next steps, the lattice path $P_n$ never back to the first row, the number of such lattice paths is $\mathcal{C}_{n-1}$. 

Let $P_n$ back to the first row in the $i^{th}$-step. So, the number of perfect lattice paths from the cell $(1,1)$ to the cell $(i-1,2)$ which staying weakly upper the line $y=1$ is $M_{i-3}$. It is remind to calculate the lattice paths from the cell $(i,0)$ to the cell $(n,n)$ of the path $P_n$ that is equal $\mathcal{C}_{n-i+1}$. We have
\begin{eqnarray*}
\mathcal{C}_n&=&\mathcal{C}_{n-1}+\mathcal{C}_{n-1}+\sum_{i=3}^{n}M_{i-3}\mathcal{C}_{n-i+1}\\
&=&2\mathcal{C}_{n-1}+\sum_{i=3}^{n}M_{i-3}\mathcal{C}_{n-i+1},
\end{eqnarray*}
Now, by using of \ref{D(s,1)} we can write
\[\mathcal{C}_{n-1}=M_{n-2}+\sum_{i=3}^{n}M_{i-3}\mathcal{C}_{n-i+1},\]
Then
\[\mathcal{C}_{n-1}=\sum_{i=3}^{n+}M_{i-3}\mathcal{C}_{n-i+1}.\]
\end{proof}
\begin{theorem}\label{columns inner product}
Inside the $m\times n$ table, we have
\[\I_m(n)=\sum_{i=1}^m\mathcal{C}(a,i)\mathcal{C}(b,i)\]
for all $a,b\geq1$ such that $a+b=n+1$. In other words, the inner product of columns $a$ and $b$ equals $\I_m(n)$. In particular, if $n=2k-1$ is odd, then
\[\I_m(n)=\sum_{i=1}^m\mathcal{C}_{k,i}^2.\]
\end{theorem}
\begin{proof}
Every perfect lattice path crosses the column $a$ at some row, say $i$. The number of such paths equals the number $\mathcal{C}(a,i)$ of paths from the first column to the $(a,i)$-cell multiplied by the number $\mathcal{C}(n-(a-1)a,i)=\mathcal{C}(b,i)$ of paths from the last column to that cell, from which the result follows.
\end{proof}

\section{Perfect lattice paths in the plane}
In this section, we calculate the number of all perfect lattice paths from $(1,1)$-cell to $(x,y)$-cell in the whole space (not restricted to a table). We denote these lattice paths with $S(x,y)$.
\begin{theorem}\label{S(x,y)}
The number $S(x,y)$ is given by
\[S(x,y)=\sum_{r=0}^{x-1}\binom{x-1}{r}\binom{x-r-1}{\frac{x-y-r}{2}}=\sum_{d=0}^{\lfloor\frac{x-y}{2}\rfloor}\binom{x-1}{d}\binom{x-d-1}{x-y-2d}.\]
\end{theorem}
\begin{proof}
Let $P$ be a lattice path starting from the $(1,1)$-cell and ending at the $(s,t)$-cell which uses three steps $\u,\r,\d$. Then $P$ is equivalent to a word on $\u,\r,\d$ satisfying
\[|\u|_P+|\r|_P+|\d|_P=x-1\quad\text{and}\quad|\u|_P-|\d|_P=y-1,\]
which implies that $|\r|_P+2|\d|_P=x-1$. The first equality follows from choosing first $r$ letter $\r$ among $x-1$ letters and then choosing $d=(x-y-r)/2$ letter $\d$ from $x-r-1$ remaindered letters, while the second equality follows from choosing first $d$ letter $\d$ among $x-1$ letter and then choosing $r=x-y-2d$ letter $\r$ from $x-d-1$ remaindered letters.
\end{proof}

Let $T=T_{m,n}$ and $S_{(a,b)}(x,y)$ denote the number of all perfect lattice path from $(a,b)$-cell to $(x,y)$-cell without leaving the table $T$ .
\begin{table}[h!]
\begin{center}
\begin{tabular}{|c|c|c|c|c|c|c|}
\hline
$6$&&&&&&$1$\\
\hline
$5$&&&&&$1$&$5$\\
\hline
$4$&&&&$1$&$4$&$15$\\
\hline 
$3$&&&$1$&$3$&$10$&$30$\\
\hline
$2$&&$1$&$2$&$6$&$16$&$45$\\
\hline 
$1$&$1$&$1$&$3$&$7$&$19$&$51$\\
\hline
$0$&&$1$&$2$&$6$&$16$&$45$\\
\hline
$-1$ &&&$1$&$3$&$10$&$30$\\
\hline
$-2$&&&&$1$&$4$&$15$\\
\hline
$-3$&&&&&$1$&$5$\\
\hline
$-4$&&&&&&$1$\\
\hline
\end{tabular}
\end{center}
\caption{Some values of $S(x,y)$ }
\label{Table:S}
\end{table}

As before one can compute $S_{(a,b)}(x,y)$ by subtracting the number of all those paths starting from $(a,b)$-cell and ending at $(x,y)$-cell and leave the table from the total number of such paths. We have
\begin{theorem}\label{S_(a,b)(x,y)}
Inside the $m\times n$ table, we have
\begin{align*}
S_{a,b}(x,y)=&S(x-a+1,y-b+1)\\
&-\sum_{x'=a+b}^{x-y}\mathcal{C}^1(x'-a,b)S(x-x'+1,y+1)\\
&-\sum_{x'=m+a-b+1}^{x+y-m-1}\mathcal{C}^1(x'-a,m-b+1)S(x-x'+1,m-y),
\end{align*}
where $a\leq x$ and $b\leq y$.
\end{theorem}
\begin{proof}
Le $T:=T_{m,n}$. Clearly, the number of $\{\u,\r,\mathcal{C}\}$-paths starting from $(a,b)$-cell and ending at $(x,y)$-cell equals $S(x-a+1,y-b+1)$ after a suitable shift. Now we count the number of those paths leaving $T$. Every such path leaves $T$ from the first row or the last row. Suppose a path leaves $T$ at $(x',0)$-cell for the first times. Clearly, $a+b\leq x'\leq x-y$, and the number of such paths equals the number of paths inside $T$ (in reverse direction) starting from $(x'-1,1)$-cell and ending at $(a,b)$-cell (namely, $\mathcal{C}^1(x'-a,b)$) multiplied by the number of paths starting at $(x',0)$ and ending at $(x,y)$-cell (namely, $S(x-x'+1,y+1)$). A similar argument shows that if a path $P$ leaves $T$ at $(x',m+1)$ for the first times, then $m+a-b+1\leq x'\leq x+y-m-1$and the number of those paths equals $\mathcal{C}^1(x'-a,m-b+1)S(x-x'+1,m-y)$, as required.
\end{proof}
\begin{example}
Utilizing Table \ref{Table:S} and Theorems \ref{S(x,y)} and \ref{S_(a,b)(x,y)}, we observe that, inside the $8\times 8$ table,
\begin{align*}
S_{(2,1)}(7,3)=&S(6,3)\\
&-\sum_{x'=3}^4\mathcal{C}^1(x'-2,1)S(8-x',4)-\sum_{x'=10}^1\mathcal{C}^1(x'-2,8)S(8-x',5)\\
=&S(6,3)-\mathcal{C}^1(1,1)S(5,4)-\mathcal{C}^1(2,1)S(4,4)\\
=&30-1\times 4-1\times 1=25.
\end{align*}
\end{example}

Using Theorem \ref{S_(a,b)(x,y)}, we can obtain a formula for $\I_m(n)$ in the case where $m+3\leq n\leq 2m+5$.
In the following, we obtain $\I_m(n)$ in the more general case that $n\geq 2m$.
\begin{theorem}
Inside table $T_{m,km+j}$, where $0\leq j\leq m-1$ and $k\geq2$, we have
\[\mathcal{I}_m(km+j)=\sum_{1\leq h_1,h_2,\ldots,h_k \leq m}\prod _{t=1}^{k-1} S_{(t,h_t)}(m+1,h_{t+1})\mathcal{C}(m,h_1)\mathcal{C}(j+1,h_k).\]
where, $k\leq \sum_{i=1}^k h_i\leq km ;\quad 1\leq h_i\leq m$.  
\end{theorem}
\begin{proof}
For positive integers $k$ and $0\leq j\leq m-1$, let $T$ be a table with $m$ rows and $n=km+j$ columns. For $0\leq i\leq m$, let $h_i$ be a cells in column $m$. The number of all perfect lattice paths from the first columns to the cell $(i,m)$ is equal $\mathcal{C}(m,h_1)$ and the number of perfect lattice paths from the cells $(h_k,mk)$ to the last column is equal $\mathcal{C}(j+1,h_k)$, for positive integers $k$  $k$ and $0\leq j\leq m-1$.
Now, the number of ways to arrive from the cells $h_i,i\times m)$ to the cells $(h_{i+1},(i+1)m)$ is $S_{(i,h_i)}(m+1,h_{i+1})$. So, the number of all perfect lattice paths from the first column to the last columns in the $T$ with $m$ rows and $n=km+j$ columns is equal
\[\mathcal{I}_m(km+j)=\sum_{1\leq h_1,h_2,\ldots,h_k \leq m}\mathcal{C}(m,h_1)\mathcal{C}(j+1,h_k)\prod _{t=1}^{k-1} S_{(t,h_t)}(m+1,h_{t+1}).\]
\end{proof}

\end{document}